\documentclass[10pt,a4paper,reqno]{amsart}
\usepackage{amsfonts}
\usepackage{amssymb}
\usepackage{amsmath}
\usepackage{amsthm}
\usepackage{cite}
\usepackage{enumitem}
\usepackage{tikz}
\usepackage{subfigure}
\usepackage{stackrel}
\usepackage{hyperref}
\usepackage{dsfont}
\theoremstyle{plain}
\newtheorem{thm}{Theorem} 

\newtheorem{lem}[thm]{Lemma}
\newtheorem{prop}[thm]{Proposition}

\theoremstyle{remark}
\newtheorem{rem}[thm]{Remark}

\def\XXint#1#2#3{{\setbox0=\hbox{$#1{#2#3}{\int}$} 
     \vcenter{\hbox{$#2#3$}}\kern-.5\wd0}}

\providecommand{\ind}{\mathds{1}} 
\providecommand{\sm}{\setminus}
\providecommand{\N}{\mathbb{N}}
\providecommand{\R}{\mathbb{R}} 
 
\providecommand{\C}{\mathbb{C}}

\providecommand{\eps}{\varepsilon}

\providecommand{\wto}{\rightharpoonup}

\DeclareMathOperator{\loc}{loc}

\DeclareMathOperator{\supp}{supp}

\DeclareMathOperator{\diam}{diam}
\DeclareMathOperator{\id}{id}
\DeclareMathOperator{\curl}{curl}

\DeclareMathOperator{\dist}{dist}

\renewcommand{\qed}{\hfill $\Box$}

\usepackage{color}
\definecolor{Darkgblue}{rgb}{0.3,0.3,0.5}

\begin{document}

\allowdisplaybreaks

\title{Ground states for Maxwell's equations in nonlocal nonlinear media}

\author{Rainer Mandel\textsuperscript{1}}
\address{\textsuperscript{1}Karlsruhe
Institute of Technology, Institute for Analysis, Englerstra{\ss}e 2, 76131 Karlsruhe, Germany}
%\email{lucrezia.cossetti@kit.edu; rainer.mandel@kit.edu}
  
%35B45: A priori estimates in the context of PDEs
%35J05: Laplace operator, Helmholtz equation (reduced wave equation), Poisson equation
%35Q61: Maxwell equations

\date{\today}   

\begin{abstract}
  In this paper we investigate the existence of ground states and dual ground states for Maxwell's Equations
  in $\R^3$ in  nonlocal nonlinear metamaterials. We prove that several nonlocal models admit ground states in
  contrast to their local analogues.
\end{abstract}
 
\maketitle
\allowdisplaybreaks
\setlength{\parindent}{0cm}

\section{Introduction}

  %This includes  stationary Schr\"oding#er
  %equations~\cite{BerLio,Strauss}  the Choquard(-Pekar) equation \r{Lieb; Moroz}

  The existence of ground states is of central importance for a large number of linear and nonlinear
  time-independent models in physics. The governing idea is that the physically most relevant nontrivial
  solution of a given PDE   is the one with  lowest  energy. Such a  solution is  called a ground
  state. In this paper we are interested in ground states for the nonlinear Maxwell
  equations
  \begin{equation}\label{eq:Maxwell-time}
	\partial_t \mathcal D-\nabla \times \mathcal H= 0, \qquad
	\partial_t \mathcal B+\nabla \times 	\mathcal E=0, \qquad
	\nabla \cdot \mathcal D=\nabla \cdot \mathcal B=0,
  \end{equation}
  that describe the propagation of electromagnetic waves in optical media without charges and currents. The
  symbols $\mathcal E,\mathcal D:\R^3\to\C^3$ denote the electric field and the electric induction whereas  
  $\mathcal H,\mathcal B:\R^3\to\C^3$ represent the magnetic field and the magnetic induction,
  respectively.
  This overdetermined system is accompanied with constitutive relations that provide a link between these
  quantities. In homogeneous and isotropic media it is usual to assume 
  $$
    \mathcal D=\eps \mathcal E + \mathcal P
    \quad\text{and}\quad 
    \mathcal B= \mu \mathcal H
  $$ 
  where $\eps,\mu\in\R\sm\{0\}$ and $\mathcal P$ denotes the so-called polarization field. In
  \cite[Section~1.3]{BaDoPlRe_GroundStates} it was shown that a time-harmonic ansatz for the electric
  field $\mathcal E(x,t)= e^{i\omega t}E(x)$ and $\mathcal P(x,t)=|\mathcal E(x,t)|^{2q-2}\mathcal E(x,t)$
  leads, after a suitable rescaling, to a nonlinear curl-curl equation of the
  form
  \begin{equation} \label{eq:Local}
    \nabla\times  \nabla\times E + E =  |E|^{2q-2} E \qquad\text{in }\R^3
  \end{equation}
  provided that $\eps\mu\omega^2$ is negative. This assumption does not hold in natural propagation
  media like vacuum, glass or water, but it holds for certain artificially produced metamaterials where
  $\eps\mu$ can be negative~\cite{SmithVierKrollSchultz}. It is therefore reasonable to investigate the
  existence of ground states for~\eqref{eq:Local} in order to single out physically relevant  
  solutions of nonlinear Maxwell equations. The energy functional associated with \eqref{eq:Local} is given by 
  $$
   \mathcal I(E) := \frac{1}{2} \int_{\R^3} |\nabla\times E|^2+|E|^2\,dx 
   - \frac{1}{2q} \int_{\R^3}  |E|^{2q}\,dx 
  $$
  for $E\in H:=H^1(\curl;\R^3)=\{ E\in L^2(\R^3;\R^3): \nabla\times E\in L^2(\R^3;\R^3)\}$. So a ground state
  is a nontrivial solution of the equation $\mathcal I'(E)=0$ such that $\mathcal I(E)$ is smallest possible
  among all nontrivial solutions. Our first observation is that ground states for this system do not exist,
  which is in striking contrast to the rich theory of ground states in the context of stationary nonlinear
  Schr\"odinger equations of the form $-\Delta u+u=|u|^{2q-2}u$ in $\R^N$~\cite{BerLio,Strauss}. Borrowing
  ideas from \cite[Theorem~1.1]{BaDoPlRe_GroundStates}, we get the following.

  \begin{prop} \label{prop:Local}
    Assume  $1<q<\infty$. Then~\eqref{eq:Local} does not have a ground state solution in $H^1(\curl;\R^3)\cap
    L^{2q}(\R^3;\R^3)$.
  \end{prop}
  
  This is surprising given that the existence of ground state solutions can be proved for
  the slightly different model where $\nabla\times  \nabla\times E + E$ is replaced by 
  $\nabla\times\nabla\times E$ and the nonlinearity is $\min\{|E|^p,|E|^q\}$ for $2<p<6<q<\infty$,
  see~\cite{Med_GS}. One may check  that the proof of Proposition~\ref{prop:Local} does not carry over to 
  this case. We also mention the cylindrically symmetric approaches from
  \cite{AzzBenDApFor_Static,DApSic_Magnetostatic,BaDoPlRe_GroundStates,HirRei_Cylindrical} where
  divergence-free solutions with higher energy are constructed. Other prototypical nonlinear Maxwell equations
  will be commented on in Remark~\ref{rem:prop}.
  %The related nonlinear Maxwell
  %equation $\nabla\times  \nabla\times E + E =  -|E|^{2q-2} E$ does not admit any nontrivial solution in
  % $H^1(\curl;\R^3)$ (test with $E$).
  % and  $\nabla\times  \nabla\times E - E =  -|E|^{2q-2} E$ does not
  %admit ground states, see Remark~\ref{rem:prop}. \r{$\nabla\times  \nabla\times E - E =    |E|^{2q-2} E$}

  \medskip
   
   Motivated by Proposition~\ref{prop:Local} our aim is to identify nonlocal variants of these nonlinear
   models that admit ground state solutions, thus overcoming the lack of local compactness in the gradient
   part of~\eqref{eq:Local}. In several physically relevant models, a nonlocal nonlinear effect 
   is given in terms of a rapidly decaying convolution kernel $K:\R^3\to\R$ that quantifies the
   dependency of the nonlinear refractive index at a given point on the intensity of the electric field over
   a small neighbourhood.
   %In such a  case the propagation medium is said to give a nolocal response, and the effect of spatial
   %dispersion is closely linked to it. 
   The kernels are typically supposed to decay rapidly at infinity. Typical choices are
   given by (see \cite[Section IV]{KrolBang})
   $$
     K_1(x)=e^{-|x|^2},\qquad K_2(x)=e^{-|x|},\qquad  K_3(x)= \ind_{|x|<R},
   $$ 
   while the latter is usually considered as a toy model allowing for explicit computations. In some cases  
   oscillatory kernel functions are considered as well~\cite{KrolBang2}.  To study the impact of such
   regularizations for the existence of ground states we first consider some partially nonlocal version
   of the cubic Kerr nonlinearity given by $(K\ast |E|^2)E$ instead of $|E|^2E$. Here, $K\ast |E|^2$
   represents a  nonlocal nonlinear refractive index change of the propagation medium.  This phenomenological
   model is particularly popular in the study of laser beams modeled by nonlinear Schr\"odinger equations,
   which in turn serve as reduced models for Maxwell's Equations \cite{KrolBang}.  
   It has the advantage of being variational, so  it makes sense to look for ground state solutions.
   However,   we show that this particular model admits ground states only under
   artificial assumptions on the kernel function $K$. In particular, we will see that $K_1,K_2$ do not admit
   ground states whereas   $K_3$ has such solutions. To be more precise, we
   introduce the associated energy functional $I:H^1(\curl;\R^3)\to\R$ via
  \begin{align} \label{eq:defI}
%     \begin{aligned}
%    I(E)
%    &:= \frac{1}{2} \int_{\R^3} \mu(x)^{-1}(\nabla\times E)(x)\cdot (\nabla\times E)(x) +   V(x)E(x)\cdot E(x)
%    \,dx \\   &
%    - \frac{1}{4} \int_{\R^3}\int_{\R^3}K(x-y) \big( V(x)E(x)\cdot E(x)\big)\,\big( V(y)E(y)\cdot E(y)\big)
%    \,dx\,dy
%    %&= \frac{1}{2} \int_{\R^3} |\nabla\times E|^2+  V(x)E\cdot E \,dx -
%    %\frac{1}{4} \int_{\R^3} [K\ast (\gamma|E|^2)]\gamma |E|^2 \,dx  \\
%     \end{aligned}
   I(E)
   &:= \frac{1}{2} \int_{\R^3} |\nabla\times E|^2 + |E|^2 \,dx - \frac{1}{4} \int_{\R^3} (K\ast |E|^2)\,|E|^2
   \,dx
 \end{align}
  where $K\in L^\infty(\R^3)$. To prove the existence of ground states the typical strategy is to
  minimize $I$ over the Nehari manifold $\mathcal N = \{ E\in H : I'(E)[E]=0, E\neq 0 \}$. The following result shows that
  this approach fails in most cases.

  \begin{thm}\label{thm:ENonlocal}
    Let $K\in L^\infty(\R^3)$ be almost everywhere continuous with $K(z)\to K(0)= \sup_{\R^3}K>0$ as $|z|\to
    0$.  Then we have
   % \begin{equation}\label{eq:infJ}
      $\inf_{\mathcal N} I 
      %= \inf_{\R^3} 
      %\lambda_{\min}\left(V^{\frac{1}{2}}\Gamma^{-\frac{1}{2}}\right)^2 K(0)^{-1}
      =    \frac{1}{4K(0)}$
      %=  \frac{1}{4}\left(\min_{\R^3} \min_{\xi\in\R^3\sm\{0\}} \frac{V(x)\xi\cdot
      %\xi}{\Gamma(x)\xi\cdot\xi}\right)^2 K(0)^{-1}.
    and a minimizer for $\inf_{\mathcal N} I$ exists if and only if $\delta_K:= \sup\{\delta: K(z)=K(0) \text{
    for }|z|<\delta\}$ is positive. In this case the set of minimizers consists of all gradient vector fields
    $E=\nabla \Phi\in\mathcal N$ with $\diam(\supp(E))\leq \delta_K$. Every such minimizer is a ground state
    solution of
    \begin{equation} \label{eq:ENonlocal}
%    \nabla\times \mu(x)^{-1} \nabla\times E + V(x) E = \big[ K\ast (V(x)E\cdot E)\big] V(x)E \qquad\text{in
%    }\R^3.
    \nabla\times \nabla\times E +   E = \big(K\ast |E|^2\big)\, E \qquad\text{in }\R^3.
   \end{equation}
  \end{thm}

%   \r{Riesz potential $K(z)=|z|^{-1}$ but $\nabla\times\nabla\times+1$ replaced with
%   $-\Delta$ \cite[Theorem~8]{Lieb}}
%   As explained above, the condition $\delta_K>0$ is typically not assumed to hold in
%   applications. So Theorem~\ref{thm:ENonlocal} indicates that generically ground states do not exist for
%   nonlocal Kerr nonlinearities. On the other hand, $K_3$ admits a non-compact uncountable set of ground states
%   all of which have compact support.
%   One also finds that the nonlocal-to-local transition is not expected to give anything interesting. In fact,
%   it is natural to consider $K\in L^1(\R^3)$ and $K_\eps(z)=\eps^{-3}K(\eps^{-1}z)$ as $\eps\to 0^+$. Our
%   Theorem says that even if minimizers for the nonlocal equation exist (model \r{C for ex}), then, after translations, they have to concentrate at a single
%   point because $\delta$ becomes $\delta\eps$, which converges to zero. 
  
  From a mathematical point of view it is interesting to see what happens for other
  nonlinearities in the nonlocal term. We have a look at the case of other power-type nonlinearities 
  where $|E|^2$ is replaced by $|E|^q$. The case $q>2$ will be excluded because the assumption $E\in H:=
  H^1(\curl;\R^3)$ does not imply $E\in L^q_{\loc}(\R^3;\R^3)$, so the Euler functional is in general
  not well-defined on $H^1(\curl;\R^3)$. In fact, we will see in Remark~\ref{rem:Thmq}~(c)
  that for reasonable kernel functions (such as $K_1,K_2,K_3$) and $q>2$ ground states cannot be constructed
  via minimization over the Nehari manifold even if the latter is intersected with $L^q(\R^3;\R^3)$.  
  For this reason we concentrate on the case $q\in (1,2)$ and define the corresponding energy as follows:   
  \begin{align*}% \label{eq:defIq}
   I_q(E)
%    &:= \frac{1}{2} \int_{\R^3} \mu(x)^{-1}(\nabla\times E)(x)\cdot (\nabla\times E)(x)+  |E(x)|^2\,dx \\   
%    &
%    - \frac{1}{2q} \int_{\R^3}\int_{\R^3}K(x-y) |E(x)|^{q/2} |E(y)|^{q/2} \,dx\,dy
   &= \frac{1}{2} \int_{\R^3} |\nabla\times E|^2+  |E|^2 \,dx 
   -  \frac{1}{2q} \int_{\R^3} \big(K\ast |E|^q\big)\, |E|^q \,dx  
 \end{align*}
  The Nehari manifold is then given by $\mathcal N_q := \{E\in H^1(\curl;\R^3): I_q'(E)[E]=0, E\neq 0\}$. 
  Our next result shows that there are ground states for each of the kernels $K_1,K_2,K_3$, which is in  
  contrast with Theorem~\ref{thm:ENonlocal}. We will assume $K\in L^{1/(2-q),\infty}(\R^3)$ to have a
  well-defined functional $I_q:L^{2/q}(\R^3)\to\R$. Here, $L^{p,s}(\R^3)$ denotes the standard Lorentz space,
  which in the case $s=\infty$ is also called  weak Lebesgue space or Marcinkiewicz space. Moreover, we assume $K$ to be Schwarz-symmetric, i.e., $K$ coincides with its
  spherical rearrangement \cite[Chapter~3]{LiebLoss}. The common shorthand notation for this   is $K=|K|^*$.
  These assumptions are satisfied if $K$ is a nonnegative radially nonincreasing function satisfying $0\leq
  K(z) \leq C|z|^{-3(2-q)}$ for some $C>0$ and almost all $z\in\R^3$, which holds for $K_1,K_2,K_3$.
  
  \begin{thm}\label{thm:ENonlocalq}
    Assume $1<q<2$, $K\in L^{1/(2-q),\infty}(\R^3)$, $K=|K|^*\not\equiv 0$ and 
    $K(\cdot+h)\to K$ in $L^{1/(2-q),\infty}(\R^3)$ as $|h|\to 0$.
    Then $\inf_{\mathcal N_q} I_q$ is attained at some ground state solution $E\in H^1(\curl;\R^3)$ of 
    $$
      \nabla\times \nabla\times E + E = \big(K\ast |E|^q\big)\, |E|^{q-2}E \qquad\text{in  }\R^3.
    $$
    All ground state solutions are irrotational and one of them is given by 
    $E(x)=t\frac{x}{|x|}f(x)^{1/q}$ where $t\neq 0$ and $f$ is a Schwarz-symmetric
    maximizer of the functional $Q(f)=\int_{\R^3} (K\ast f)f\,dx$ over the unit sphere in $L^{2/q}(\R^3)$.   
  \end{thm}
  
  %In view of Theorem~\ref{thm:ENonlocal} it seems clear that the maximizers diverge as $q\to
  %2$ in any reasonable topology. This is supported by the fact that $Q$ does not have a minimizer
  %over the unit sphere in $L^1(\R^3)$ provided that $K$ is as in Theorem~\ref{thm:ENonlocal}. 

  Irrotational vector fields satisfy $\nabla\times E=0$ in $\R^3$,
  so there is a potential $\Phi:\R^3\to\R$ such that  $E=\nabla\Phi$. Such electric fields do not generate a 
  magnetic field in view of \eqref{eq:Maxwell-time}. This is different for the nontrivial cylindrically
  symmetric solutions of these equations found in
  \cite{DApSic_Magnetostatic,BaDoPlRe_GroundStates,HirRei_Cylindrical} that are divergence-free and hence do
  not satisfy $\nabla\times E=0$. The ground states obtained by Mederski~\cite{Med_GS}
  do not have this property either, so both types of solutions come with a nontrivial magnetic field.
   
  \medskip

%   * Rectangular leads to optimizers of $\|\hat f\|_{L^2(B)} \leq C\|f\|_{2/q}$. By Diogo one knows 
%   for $2/q < \frac{2(N+1)}{N+3}=\frac{4}{3}$, i.e., $2>q>\frac{3}{2}$, $f_r(x)=r^{-3}f(r^{-1}x)$ 
%   $$
%     \|\hat f\|_{L^2(B)}^2
%     = \int_0^\rho  r^2 \|\hat f(r\cdot)\|_{L^2(S^2)}^2 \,dr
%     = \int_0^\rho  r^2 \|\hat f_r\|_{L^2(S^2)}^2 \,dr
%     \leq \int_0^\rho   r^2 C_{ST} \|f_r\|_{L^{2/q}}^2 \,dr
%     = \int_0^\rho   r^2 C_{ST} r^{-6+3q} \|f\|_{L^{2/q}}^2 \,dr
%     = \frac{C_{ST}}{3(q-1)} \rho^{3(q-1)} \|f\|_{L^{2/q}}^2 
%   $$
%   with equality iff $\hat f_r(\omega)=1$ for almost all $r$ and almost all $\omega\in S^2$. In other words,
%   the minimizer is $\mathcal F^{-1}(1_B)$
    
  Given the importance of the cubic Kerr nonlinearity for nonlinear optics, we discuss another
  nonlocal model that admits nontrivial solutions also in the case of a cubic nonlinearity. This model
  originates from a description of the polarization field $\mathcal P(x,t)=P(x)e^{i\omega t}$
  via $P= K\ast (|E|^2 E)$ instead of $P= E (K\ast |E|^2)$. In the linear case such fully nonlocal models are 
  investigated  both from  a physical and mathematical point of view~\cite{GoffiPlum,Schoonover,Kinsler}.
  %, see \cite[Chapter XII]{LandLif} for an introduction reference.     
  Our aim is to show that such models often admit some sort of ground state solution, which stands in 
  contrast to Theorem~\ref{thm:ENonlocal}. 
  %As a consequence, these models need to be discussed separately from each other
  %in the context of Nonlinear Maxwell Equations.
  It even allows to treat all power-type nonlinearities
  $|E|^{2q-2}E$ with $1<q<\infty$ under appropriate assumptions on the convolution kernel~$K$. The following
  result applies to each of the kernels $K_1,K_2,K_3$ introduced above. 

%   \begin{itemize}
%     \item[(B)] $K\in L^{r/2}(\R^3)$ is nontrivial with $\hat K\geq 0$, $\hat K>0$ near zero 
%     and $|\partial^\alpha \hat K(\xi)| \leq \langle \xi\rangle^{-s-|\alpha|}$ for all multi-indices
%     $\alpha\in\N_0^3$ and $s>\frac{3(r-2)}{r}$.
%   \end{itemize}
  
  \begin{thm} \label{thm:Nonlocal} 
    Assume $1<q<\infty, K\in L^{q,\infty}(\R^3)$, $K(\cdot+h)\to K$ in $L^{q,\infty}(\R^3)$ as $|h|\to 0$
    and $\int_{\R^3} (K\ast f)f\,dx>0$ for some $f\in L^{(2q)'}(\R^3)$.  
    Then there is a dual ground state solution $E\in L^{2q}(\R^3;\R^3)$ for 
  \begin{align}\label{eq:FullyNonlocal}
    \nabla\times \nabla\times E +  E = K\ast (|E|^{2q-2}E) \qquad\text{in }\R^3. 
  \end{align}
    %If $\hat K\geq   0$ on $\R^3$, then there is even an irrotational dual ground state solution, and if even
    % $\hat K>0$ holds on $\R^3$, then each dual ground state solution is necessarily irrotational.
  \end{thm}
  
  Since \eqref{eq:FullyNonlocal} is not variational, the notion of a ground state does
  not make sense. A dual ground state is a function $E$ given by
  $|E|^{2q-2}E=U$ where $U$ is a ground state for the associated dual functional $J:L^{(2q)'}(\R^3;\R^3)\to\R$
  that we will introduce in~\eqref{eq:functionalJdual}. This functional is defined in such a way that any
  critical point $U$ of $J$ gives rise to a distributional solution $E$ of~\eqref{eq:FullyNonlocal}. In the
  Appendix we will motivate why such dual ground states may be interpreted as reasonable substitutes for
  ground states. 
%   Note that Theorem~\ref{thm:Nonlocal} applies to nonradial kernel functions. In that case, no
%   solution is a gradient of a radially symmetric function like the solutions that
%   we obtained for the partially nonlocal model from Theorem~\ref{thm:ENonlocalq}. This follows from taking
%   the Fourier transform of~\eqref{eq:FullyNonlocal}.   
  In each of our results one can deduce better  integrability and regularity
  properties of the constructed solutions under suitable assumptions on the kernel function $K$  using
  classical bootstrap arguments. Clearly, it would be of interest to know about further qualitative properties
  of ground states  like their symmetries, monotonicity and positivity properties. We believe it to be
  particularly interesting whether ground states are unique and irrotational.
  
  \section{Proof of Proposition~\ref{prop:Local}}
  
  We have to show that the equation $\nabla\times\nabla\times E + E = |E|^{2q-2}E$ in $\R^3$ does not have
  ground state solutions. To this end define $E_j(x):= \ind_{|x|<1/j} \frac{x}{|x|}$ for $j\in\N$. Then
  $E_j\in H^1(\curl;\R^3) \cap L^{2q}(\R^3;\R^3)$ is a weak solution of \eqref{eq:Local} satisfying $\nabla\times
     E_j=0$ in the weak sense. The latter is true because of $E_j = \nabla\Phi_j$ where $\Phi_j\in
     H^1_{\loc}(\R^3)$ is given by $\Phi_j(x):= \min\{|x|,\frac{1}{j}\}$.     
     From $I(E_j)\to 0$ as $j\to\infty$ we deduce that a ground state can only
     satisfy $I(E)=0$ if it exists. On the other hand, $I'(E)=0,E\neq 0$ gives $I(E)=I(E)-\frac{1}{2q}I'(E)[E]
     = \frac{q-1}{2q}\int_{\R^3} |E|^{2q}\,dx >0$, so there is no nontrivial solution at the energy level
     zero. Hence, a ground state cannot exist. \qed
   
   \medskip  
  \newpage
  
  \begin{rem} \label{rem:prop}~
    \begin{itemize}
      \item[(a)] In \cite[Theorem~1.1]{BaDoPlRe_GroundStates} the authors claim that all distributional
    solutions  $E(x)=\frac{x}{|x|}s(|x|)$ of the equation $\nabla\times\nabla\times E + E =
    |E|^{2q-2}E$ are given by arbitrary measurable functions $s:(0,\infty)\to \{-1,1\}$. However, this is
    only a subfamily of all such distributional solutions. The correct version of this result states
    that all distributional solutions  $E(x)=\frac{x}{|x|}s(|x|)$ are given by arbitrary measurable functions    
    $s:(0,\infty)\to \{-1,0,1\}$. This small modification is important, because we exploit a shrinking of supports in our proof above by choosing
    $s(|x|)=\ind_{|x|<1/j}$. In view of \cite[Theorem~1.1]{BaDoPlRe_GroundStates} Proposition~\ref{prop:Local}
    generalizes to other nonlinear Maxwell equations like  $\nabla\times\nabla\times E + E = g(|x|,|E|)E$ where shrinking solutions are given by 
    $E_j(x):= \ind_{\rho<|x|<\rho+1/j} a(|x|) \frac{x}{|x|}$ for any measurable function $a$ given by
    $g(r,a(r))=1$ for almost all $r\in (\rho,\rho+\frac{1}{j}), \rho>0$. 
    \item[(b)] Similarly, one obtains solutions with  compact support for $\nabla\times\nabla\times E - E =
    g(|x|,|E|)E$. In particular, $\nabla\times\nabla\times E - E = -|E|^{2q-2}E$ with $1<q<\infty$ has the
    solutions $E_j(x):= \ind_{0<|x|<j}  \frac{x}{|x|}$ and the associated energy 
    $$
      I(E_j)
      = \frac{1}{2} \int_{\R^3} |\nabla \times E_j|^2 - |E_j|^2\,dx + \frac{1}{2q} \int_{\R^3} |E_j|^{2q}\,dx 
      = - \frac{q-1}{2q} |\{x\in\R^3: 0<|x|<j\}|
    $$
    tends to $-\infty$ as $j\to\infty$. So ground states do not exist for these equations. 
    Moreover, $\nabla\times\nabla\times E + E = -|E|^{2q-2}E$ does not admit any nontrivial weak
    solution in $H^1(\curl;\R^3)\cap L^{2q}(\R^3;\R^3)$, which follows from testing the equation with $E$. The equation
    $\nabla\times\nabla\times E - E = |E|^{2q-2}E$ admits nontrivial cylindrically symmetric solutions for
    suitable exponents $q$ \cite[Theorem~3a]{Man_Uncountably}, but these solutions do not belong to
    $H^1(\curl;\R^3)$ due to slow decay rates at infinity. In the related case of Helmholtz
    equations  nonexistence results for $L^2$-solutions can be found in~\cite[Theorem~1a]{Kato}
    or~\cite[Theorem~3]{KoTa}. Under strong extra assumptions on $|E(x)|^{2q-2}$ the absence of
    $H^1(\curl;\R^3)$-solutions $E$ follows from \cite[Theorem~3]{CosMan}
    choosing $\mu(x):=1,\eps(x):=1-|E(x)|^{2q-2}$.  
    \item[(c)] The existence of discontinuous and concentrating solutions illustrates that no regularity theory and no compact embeddings
    in whatever Lebesgue space can be exploited in the analysis of the (local) nonlinear Maxwell
    equation~\eqref{eq:Local}. 
    \end{itemize}
  \end{rem} 
  
 \section{Proof of Theorem~\ref{thm:ENonlocal}}
    
  It is convenient to split the functional $I$ according to
   $I(E) = \frac{1}{2} I_L(E) - \frac{1}{4} I_{NL}(E)$ where 
 \begin{align*}
   I_L(E) 
    = \int_{\R^3}  |\nabla\times E|^2 + |E|^2 \,dx, \qquad  
   I_{NL} (E) 
   = \int_{\R^3}\int_{\R^3}K(x-y) |E(x)|^2|E(y)|^2  \,dx\,dy,
 \end{align*}
 see~\eqref{eq:defI}. % and \eqref{eq:defMW}. 
 It is standard to show that $I$ is continuously differentiable
 on $H^1(\curl;\R^3)$ under the assumptions of Theorem~\ref{thm:ENonlocal} with Fr\'{e}chet derivative
 \begin{align*}
   I'(E)[\tilde E] 
   = \int_{\R^3} (\nabla\times E)\cdot (\nabla\times \tilde E)  +   E\cdot
   \tilde E \,dx 
    - \int_{\R^3}\int_{\R^3}K(x-y) \big( E(x)\cdot \tilde E(x)\big)\,|E(y)|^2 \,dx\,dy
 \end{align*}
 for $E,\tilde E\in H$. 
%  In fact, $K\in L^\infty(\R^3)$ implies for $f\in L^1(\R^3)$  
%  $$
%    \int_{\R^3} f (K\ast f)\,dx
%    \leq \|f\|_1 \|K\ast f\|_\infty ^2  
%    \leq \|K\|_\infty \|f\|_1^2
%  $$
%  and plugging in $f:=\gamma |E|^2$ shows that $J$ is well-defined.
  Moreover,  critical points of $I$ are weak solutions of~\eqref{eq:ENonlocal}, so a ground state
  solution may be obtained by minimizing $I$ over the Nehari manifold 
  $\mathcal N = \{ E\in H : I'(E)[E]=0, E\neq 0 \}$. 
 We first provide a convenient min-max characterization of the least energy level $\inf_{\mathcal N} I$ with
 the aid of the fibering map $\gamma(t):= I(tE)$ for $E\in H\sm\{0\}$.
 The simple observation is that $tE\in\mathcal N$ holds for some $t\neq 0$ if and only if $\gamma'(t)=0$.
 Given the structure of $I$ it immediate to see that  $\gamma$ has a unique positive maximizer if
 $I_{NL}(E)>0$. In the opposite case, $\gamma$ increases to $+\infty$ and does not have any critical point.
 This implies
 \begin{align} \label{eq:def_c}
   \inf_{\mathcal N} I
    = \inf_{E\in H\sm\{0\}} \sup_{t>0} I(tE)
    =   \inf_{E\in H, I_{NL}(E)>0} \frac{I_L(E)^2}{4I_{NL}(E)}.    
 \end{align}
 Note that $K(z)>0$ for small $|z|$ implies that $I_{NL}(E)>0$  holds for $E$ belonging to
 some nonempty open subset of $H$. In fact, one may take $E$ with support of sufficiently small diameter. We
 conclude that the Nehari manifold is non-void and it remains to analyze the expression on the right
 of~\eqref{eq:def_c}. 
  
 \medskip
 
 We first prove the formula for $\inf_{\mathcal N}J$. The lower bound is
 obtained as follows:
  \begin{align} \label{eq:lowerbound}
    \begin{aligned}
    \frac{I_L(E)^2}{4I_{NL}(E)}
    &\geq \frac{( \int_{\R^3}  
   |E(x)|^2  \,dx)^2}{4\int_{\R^3}\int_{\R^3} K(x-y) |E(x)|^2 |E(y)|^2  \,dx\,dy} \\
    &\geq \frac{(\int_{\R^3}  
   |E(x)|^2   \,dx)^2}{4K(0) \int_{\R^3}\int_{\R^3}  |E(x)|^2 |E(y)|^2 \,dx\,dy} \\ 
    &= \frac{1}{4K(0)}.
   \end{aligned} 
  \end{align} 
  Here we used $K(z)\leq K(0)$ for all $z\in\R^3$.  
  %\b{Equality: $V(x)E(x)\cdot E(x)= R(x_0) \Gamma(x)E(x)E(x)$ almost everywhere}
 On the other hand, choosing a nonconstant $\phi\in \mathring{H}^1(\R^3)$ such that $\nabla\phi$ has compact
 support and $\tilde E_n(x):=n^{3/2} \nabla \phi(nx)$,  we observe $\nabla\times
 \tilde E_n=0$ and $\tilde E_n\in L^2(\R^3;\R^3)$. In particular, $\tilde E_n\in H=H^1(\curl;\R^3)$.
 Moreover, the supports of $\tilde E_n$ shrink to $\{0\}$. Hence, $K(z)\to K(0)$ as $z\to 0$ implies
 \begin{align*}
   \frac{I_L(\tilde E_n)^2}{4I_{NL}(\tilde E_n)}
   &= \frac{(\int_{\R^3} |\tilde E_n(x)|^2 \,dx)^2}{4\int_{\R^3}\int_{\R^3}K(x-y)
    |\tilde E_n(x)|^2  |\tilde E_n(y)|^2  \,dx\,dy}   \\
   &= \frac{ (\int_{\R^3}  |\tilde E_n(x)|^2\,dx)^2}{4  \int_{\R^3}\int_{\R^3}
   (K(0)+o(1))  |\tilde E_n(x)|^2  |\tilde E_n(y)|^2 \,dx\,dy}   \\
   &= \frac{ (\int_{\R^3} |\nabla\phi(x)|^2  \,dx)^2}{4  \int_{\R^3}\int_{\R^3}
   K(0) |\nabla\phi(x)|^2 |\nabla\phi(y)|^2  \,dx\,dy + o(1)}   \\
   &= \frac{1}{4K(0)} +o(1) \quad\text{as }n\to\infty.
 \end{align*}
 This proves the formula for the infimum.
 
 \medskip
 
 Next we show that a minimzer exists if and only if $\delta_K>0$, i.e., if and only if $K(z)=K(0)$ for almost
 all $z\in\R^3$ such that $|z|<\delta$ where $\delta>0$. In fact, choose any nonconstant $\phi\in
 \mathring{H}^1(\R^3)$ such that the support of $\tilde E:= \nabla \phi\in H, \tilde E\in L^2(\R^3;\R^3)$ has
 diameter $\leq \delta$. Then $\nabla\times\tilde E=0$ implies $\tilde E\in H$ and both  inequalities in
\eqref{eq:lowerbound} become equality. 
%  \begin{align*}
%    \frac{I_L(\tilde E)^2}{4I_{NL}(\tilde E)}
%    &= \frac{(\int_{\R^3} V(x)\tilde E\cdot \tilde E\,dx)^2}{4\int_{\R^3}\int_{\R^3}K(x-y) [\Gamma(x)\tilde
%    E(x)\cdot \tilde E(x)][\Gamma(y)\tilde E(y)\cdot \tilde E(y)] \,dx\,dy}   \\
%    &= \frac{(\int_{\R^3} R(x_0) \Gamma(x)\tilde E\cdot \tilde E\,dx)^2}{4\int_{\R^3}\int_{\R^3}K(0)
%    [\Gamma(x)\tilde E(x)\cdot \tilde E(x)][\Gamma(y)\tilde E(y)\cdot \tilde E(y)] \,dx\,dy}   \\
%    &= \frac{R(x_0)^2 (\int_{\R^3} \Gamma(x)\tilde E\cdot \tilde E\,dx)^2}{4K(0)(\int_{\R^3} 
%    \Gamma(x)\tilde E(x)\cdot \tilde E(x) \,dx)^2}   \\
%    &= \frac{R(x_0)^2}{4K(0)}
%    = \frac{(\min  R)^2}{4K(0)}.  
%  \end{align*}
%  It remains to prove that a minimizer can only occur in such a situation. The estimates \eqref{eq:lowerbound}
%  show that if $E$ attains $(\min_{\R^3} R)^2 K(0)^{-1}$ then we must have
%  \begin{itemize}
%    \item $\nabla\times E = 0$
%    \item Equality: $V(x)E(x)\cdot E(x)= R(x_0) \Gamma(x)E(x)E(x)$ almost everywhere. \\
%    $E(x)^T (V(x)-R(x_0)\Gamma(x))E(x)=0$ almost everywhere. \\
%    $E(x)^T (V(x)-R(x_0)\Gamma(x))E(x)=0$ on $supp(E)$. \\
%    \item $\{K= K(0)\}\supset 2\supp(E)$ 
%  \end{itemize} 
Vice versa, if $E$ is a minimizer, then the estimates
from \eqref{eq:lowerbound} show that $\nabla\times E = 0$ holds
almost everywhere and  $K(x-y)=K(0)$ has to hold for almost all
$(x,y)\in \supp(E)\times\supp(E)$. Since $K$ is continuous almost everywhere, this implies $K(z)=K(0)$ for
almost all $z\in\R^3$ with $|z|\leq \diam(\supp(E))$, which is all we had to prove.  \qed

%  \begin{align} 
%     \begin{aligned}
%     \frac{I_L(E)^2}{I_{NL}(E)}
%     &\geq \frac{(\int_{\R^3} V(x)E\cdot E\,dx)^2}{\int_{\R^3}\int_{\R^3}K(x-y) [\Gamma(x)E(x)\cdot 
%     E(x)][\Gamma(x)E(y)\cdot E(y)] \,dx\,dy} \\
%    &\geq \frac{(\min R \int_{\R^3} \Gamma (x)E\cdot E\,dx)^2}{K(0)\int_{\R^3}\int_{\R^3}
%    [\Gamma(x)E(x)\cdot \tilde E(x)][\Gamma(x)E(y)\cdot E(y)] \,dx\,dy} \\
%    &= (\min  R)^2 K(0)^{-1}.
%    \end{aligned} 
%   \end{align}   
%  

 \section{A Lemma}
   
   In this section we consider an auxiliary variational problem that will allow to deduce the existence 
   of solutions for the nonlocal nonlinear Maxwell equations that we discuss in Theorem~\ref{thm:ENonlocalq}
   and Theorem~\ref{thm:Nonlocal}. It shares some features with Lions' application of the
    concentration-compactness principle presented in \cite[Section II.1-2]{Lions_ConcComp}. Let $S:=\{f\in
   L^p(\R^N;\R^M):\|f\|_p=1\}$ denote the unit sphere. Our aim is to solve the maximization problem 
   $$ 
      \sup_{f\in S} Q(f)\qquad\text{where } Q(f):= \int_{\R^N} (\mathcal K\ast f) \cdot f\,dx 
   $$
   under appropriate assumptions on the kernel function $\mathcal K$. To prove the
   existence of a maximizer, we will use Lions' concentration-compactness method. So we consider a
   maximizing sequence  $(f_j)$ in $S$ and define the probability measures 
   $$
    \mu_j(B) :=  \int_B |f_j(x)|^p\,dx.
  $$
  According to \cite[Theorem~4.7.3]{KCChang}, see \cite[Lemma~I.1]{Lions_ConcComp} for the original result,
  this family of measures may behave in three possible ways:
  \begin{itemize}
    \item[(I)] (Compactness) There is a sequence $(x_j)\subset\R^N$ such that for all $\eps>0$ there is
    $R_\eps>0$ such that $\mu_j(B_{R_\eps}(x_j))\geq 1-\eps$.
    \item[(II)] (Vanishing) For all $R>0$ we have $\lim_{j\to\infty} \sup_{x\in\R^N} \mu_j(B_R(x))=0$.
    \item[(III)] (Dichotomy) There are $\lambda\in (0,1)$, a sequence $(R_j)$ with $R_j\to\infty$, a sequence
    $(x_j)\subset \R^N$ and measures $\mu_j^1,\mu_j^2$ such that $0\leq \mu_j^1+\mu_j^2\leq \mu_j$ and
    \begin{equation*}%\label{eq:dichotomy}
      \supp(\mu_j^1) \subset B_{R_j}(x_j),\quad
      \supp(\mu_j^2) \subset \R^N\sm B_{2R_j}(x_j),\quad
      |\lambda-\mu_j^1(\R^N)| + |1-\lambda-\mu_j^2(\R^N)| \leq \frac{1}{j}.
    \end{equation*}
  \end{itemize}
  The aim is to show that (I) occurs and to derive the existence of a maximizer using some
  local compactness property of $f\mapsto \mathcal K\ast f$. The first step is to rule out the vanishing case
  (II). This is achieved with the aid of the following simple estimate.
  
  \begin{prop}\label{prop:nonvanishing}
    Assume $N,M\in\N, 1<p<2$ and $\mathcal K\in L^{p'/2,\infty}(\R^N;\R^{M\times M})$. Then we have
    for all $R>0$ $$
      Q(f)
      \leq \eps(R) \|f\|_p^2 + C(R) \|f\|_p^p  \left(\sup_{y\in\R^N}
      \|f\|_{L^p(B_R(y))}\right)^{2-p}
    $$
    with $\eps(R)\to 0$ as $R\to\infty$.
  \end{prop}
  \begin{proof}
    We have 
    $$
      Q(f)
      = \int_{\R^N} (\mathcal K\ast f)\cdot f\,dx
      = \int_{\R^N} ((\mathcal K- \mathcal K_R)\ast f)\cdot f\,dx 
      + \int_{\R^N} (\mathcal K_R\ast f) \cdot f\,dx.
    $$
    where $\mathcal K_R(z):= \mathcal K(z)\ind_{|z|+|\mathcal K(z)|\leq R}$ . Then Young's convolution
    inequality from~\cite[Theorem~1.4.25]{Graf} gives due to $2<p'<\infty$
    \begin{align*}
       \int_{\R^N} ((\mathcal K- \mathcal K_R)\ast f)\cdot f\,dx
       \leq \|\mathcal K-\mathcal K_R\|_{\frac{p'}{2},\infty}  \|f\|_p^2
    \end{align*}
     and the prefactor goes to zero as $R\to \infty$ by the Dominated Convergence Theorem.
     On the other hand, the estimates from \cite[p.109-110]{AlamaLi} yield
    \begin{align*}
      \int_{\R^N} (\mathcal K_R\ast f) \cdot f\,dx
      &\leq  \|\mathcal K_R\|_\infty \int_{\R^N} \int_{\R^N} \ind_{|x-y|\leq R}
        |f(x)||f(y)| \,dx\,dy \\
      &\leq C(R)  \|f\|_p^p \left(\sup_{y\in\R^N}
      \|f\|_{L^p(B_{3\sqrt{3}R}(y))}\right)^{2-p}
    \end{align*}
    and the claim follows.
 \end{proof}

 \begin{prop}\label{prop:RuleOut}
   Assume $N,M\in\N, 1<p<2, \mathcal K\in L^{p'/2,\infty}(\R^N;\R^{M\times M})$ and assume that
   $(f_j)\subset S$ is a maximizing sequence for $m := \sup_{f\in S} Q(f)>0$  
   with induced measures $\mu_j$. Then neither (II) nor (III) occurs.
 \end{prop}
 \begin{proof}
    For large $j$ we have due to Proposition~\ref{prop:nonvanishing}  
    $$
      \frac{m}{2}
      \leq Q(f_j)
      \leq \eps(R)  + C(R) \left(\sup_{y\in\R^N} \|f_j\|_{L^p(B_R(y))}\right)^{2-p}. 
    $$
    with $\eps(R)\to 0$ as $R\to\infty$. Hence, we may choose $R>0$ so large  that 
    $\sup_{y\in\R^N}\mu_j(B_R(y))>0$ holds. So the case (II) cannot occur. Now assume (III) and choose
    $(x_j),(R_j),\lambda\in (0,1)$ accordingly.  We decompose the sequence according to $f_j =
    f_j^1+f_j^2+f_j^3$ where $f_j^1:= f_j\ind_{B_{R_j}(x_j)}$ and $f_j^2:= f_j\ind_{\R^N\sm B_{2R_j}(x_j)}$.
    %,  $f_j^3:= f_j\ind_{A_j}$ where $A_j:=B_{2R_j}(x_j)\sm B_{R_j}(x_j)$. 
    Then (III) implies
   \begin{equation}\label{eq:NoDichotomy1}
      \int_{\R^N} |f_j^1|^p\,dx  \to \lambda,\qquad
      \int_{\R^N} |f_j^2|^p\,dx   \to 1-\lambda,\qquad
      \int_{\R^N} |f_j^3|^p\,dx\to 0.\qquad (j\to\infty)
   \end{equation} 
%    \begin{equation}\label{eq:NoDichotomy1}
%       \int_{\R^3} |f_j^3|^p
%       = \mu_j(A_j)
%       \leq  \mu_j(A_j) + (\mu_j-\mu_j^1-\mu_j^2)(\R^3\sm A_j)
%       = \mu_j(\R^3) - \mu_j^1(\R^3)-\mu_j^2(\R^3)
%       \to  0
%        \qquad\text{as }j\to\infty.
%    \end{equation}
%    From this we deduce
%    \begin{equation}\label{eq:NoDichotomyII1}
%       \int_{\R^3} |f_j^1|^p\,dx = \mu_j(B_{R_j}(x_j)) \to \lambda,\qquad
%       \int_{\R^3} |f_j^2|^p\,dx = \mu_j(\R^3\sm B_{2R_j}(x_j)) \to 1-\lambda.
%    \end{equation}
   Then $\|f_j^3\|_p\to 0$ and $\dist(\supp(f_j^1),\supp(f_j^2))\geq R_j\to \infty$ as $j\to\infty$  yield
   \begin{align*}
    Q(f_j)
     &=  \int_{\R^N} (\mathcal K\ast f_j)\cdot f_j \,dx \\
     &=  \int_{\R^N} (\mathcal K\ast f_j^1)\cdot  f_j^1  \,dx
      +  \int_{\R^N} (\mathcal K\ast f_j^2)\cdot  f_j^2  \,dx
      +  \int_{\R^N} (\mathcal K\ast f_j^3)\cdot  f_j^3  \,dx \\
      &+  2 \int_{\R^N} (\mathcal K\ast f_j^2)\cdot  f_j^1  \,dx 
       + 2  \int_{\R^N} (\mathcal K\ast f_j^3)\cdot  (f_j^1+f_j^2) \,dx  \\
     &\leq  Q(f_j^1) + Q(f_j^2)
      +  \|\mathcal K\|_{\frac{p'}{2},\infty}  \|f_j^3\|_p^2  \\
      &+ 2  \|\mathcal K\ind_{\R^N\sm B_{R_j}(0)}\|_{\frac{p'}{2},\infty} \|f_j^2\|_p \|f_j^1\|_p 
       + 2  \|\mathcal K\|_{\frac{p'}{2},\infty}  \|f_j^3\|_p \|f_j^1+f_j^2\|_p  \\  
      &=  Q(f_j^1) + Q(f_j^2) + o(1)  \qquad\text{as }j\to\infty.
   \end{align*}
   Furthermore, since $f_j^1,f_j^2$ are nontrivial for large $j$, the definition of $m$ implies
   \begin{align*}
     m
     \geq  Q(f_j^1/\|f_j^1\|_p)
     \stackrel{\eqref{eq:NoDichotomy1}}=  Q(f_j^1)\lambda^{-\frac{2}{p}} + o(1),\qquad
     m
     \geq    Q(f_j^2/\|f_j^2\|_p)
     \stackrel{\eqref{eq:NoDichotomy1}}=    Q(f_j^2)(1-\lambda)^{-\frac{2}{p}} + o(1).
   \end{align*}
   Combining the previous estimates   we get
   \begin{align*}
      m
     = Q(f_j) + o(1)
     \leq  Q(f_j^1) + Q(f_j^2) + o(1) 
     \leq m\cdot \left(\lambda^{\frac{2}{p}}+(1-\lambda)^{\frac{2}{p}}\right) + o(1),
   \end{align*}
   which is impossible due to $\lambda\in (0,1)$ and $p<2$. So (III) cannot occur either.
 \end{proof}
 
   So we are left with the compactness case (I). So we exploit some local compactness property to deduce the
   existence of a maximizer. This is provided next.
   
   \begin{prop}\label{prop:LocalCompactness}
     Assume $N,M\in\N, 1<p<2$ and that $\mathcal K\in L^{p'/2,\infty}(\R^N;\R^{M\times M})$ satisfies
     $\mathcal K(\cdot+h)\to \mathcal K$ in $L^{p'/2,\infty}(\R^N;\R^{M\times M})$ as $|h|\to 0$. Then
     $L^p(\R^N;\R^M)\to L^{p'}(B;\R^M), f\mapsto \mathcal K\ast f$ is compact for all bounded sets $B\subset\R^N$.
   \end{prop}
   \begin{proof}
     We use the Fr\'{e}chet-Kolmogorov-Riesz criterion \cite[Theorem~4.26]{Brezis} that characterizes precompact subsets in
     Lebesgue spaces with exponent $<\infty$. Being given any bounded sequence $(f_j)$ in $L^p(\R^N;\R^M)$ we
     have to show that $\{\mathcal K\ast f_j:j\in\N\}$ is precompact in $L^{p'}(\R^N;\R^M)$. The estimate $\|\mathcal K\ast
     f_j\|_{p'}\leq \|K\|_{p'/2,\infty}\|f_j\|_p$ shows that the family is bounded. Moreover,
     $$
       \|(\mathcal K\ast f_j)(\cdot+h)- (\mathcal K\ast f_j)\|_{p'}
       = \|(\mathcal K(\cdot+h)-\mathcal K)\ast f_j\|_{p'}
       \leq \|\mathcal K(\cdot+h)-\mathcal K\|_{p'/2,\infty} \|f_j\|_p.
     $$
     So the boundedness of $(f_j)$ and  $\|\mathcal K(\cdot+h)-\mathcal K\|_{p'/2,\infty}\to 0$ as $|h|\to 0$
     imply the equicontinuity of $\{\mathcal K\ast f_j:j\in\N\}$ in $L^{p'}(\R^N;\R^M)$. By the above-mentioned criterion, this implies that  $\{\mathcal K\ast f_j:j\in\N\}$ is
     precompact in $L^{p'}(B;\R^M)$ for all bounded sets $B\subset\R^N$.
   \end{proof}

  \begin{lem} \label{lem:MinimizerQ}
    Assume $N,M\in\N, 1<p<2$ and that $\mathcal K\in L^{p'/2,\infty}(\R^N;\R^{M\times M})$ satisfies 
    $\mathcal K(\cdot+h)\to \mathcal K$ in $L^{p'/2,\infty}(\R^N;\R^{M\times M})$ as $|h|\to\infty$ as well as
    $m:=\sup_{f\in S} Q(f) > 0$.
    Then the functional $Q$ has a maximizer over $S$. Moreover:
    \begin{itemize}
      \item[(i)]  If each entry of $\mathcal K$ is a nonnegative function, then $Q$ has a componentwise
      nonnegative maximizer.
      \item[(ii)] If $M=1$ and $\mathcal K=|\mathcal K|^*$, then $Q$ has a Schwarz-symmetric maximizer. 
      \item[(iii)] If $M=1$ and $\mathcal K=|\mathcal K|^*$ is radially decreasing, then each maximizer of $Q$
      is Schwarz-symmetric up to translations.
    \end{itemize} 
  \end{lem}
  \begin{proof}  
    Let $(f_j)\subset S$ be a maximizing sequence for $m>0$. Then the
    Concentration-Compactness Lemma (see above) and Proposition~\ref{prop:RuleOut} imply that the sequence of measures $(\mu_j)$ induced by $(f_j)$ satisfies alternative (I). Since 
    $Q$ and $S$ are translation-invariant, we may assume $x_j=0$ in (I). Since $(f_j)$ is
    bounded in the reflexive Banach space $L^p(\R^N;\R^M)$, we may furthermore assume $f_j\wto f$ in
    $L^p(\R^N;\R^M)$ where  $\|f\|_p \leq \liminf_{j\to\infty} \|f_j\|_p = 1$. 
   In order to show that $f$ maximizes $Q$ let $\eps>0$ be arbitrary and choose $R_\eps>0$ as in (I). Then we
   have 
   $$
     \int_{\R^N\sm B_{R_\eps}(0)} |f_j|^p\,dx \leq \eps.
    $$
   Hence,
 \begin{align*}
   Q (f_j)
   &= \int_{\R^N}  (\mathcal K\ast f_j)\cdot f_j \,dx \\
   &\leq  \int_{B_{R_\eps}(0)}  (\mathcal K \ast f_j)\cdot f_j \,dx
    + \|\mathcal K\|_{\frac{p'}{2}} \|f_j\|_{p} \|f_j \ind_{\R^N\sm B_{R_\eps}(0)}\|_{p} \\
   &\leq   \int_{B_{R_\eps}(0)} (\mathcal K\ast f_j)\cdot f_j  \,dx  +  \|\mathcal K\|_{\frac{p'}{2}}  
   \eps^{\frac{1}{p}}.
 \end{align*}
   We have  $f_j\wto f$ in $L^p(\R^N)$ and Proposition~\ref{prop:LocalCompactness} implies
   $\mathcal K\ast  f_j \to \mathcal K \ast f$ in $L^{p'}(B_{R_\eps}(0))$.  Hence, passing to the limit $j\to\infty$ we find
   $$
     m = \lim_{j\to\infty} Q(f_j) \leq  \int_{B_{R_\eps}(0)} (\mathcal K\ast f)\cdot f  \,dx  + 
     \|\mathcal K\|_{\frac{p'}{2}}   \eps^{\frac{1}{p}}.  
   $$ 
   Using $\|f\|_p \leq  \liminf_{j\to\infty} \|f_j\|_p =1$
   and sending $\eps$ to zero, we infer from the Dominated Convergence Theorem
   $$
      0 < m  = \lim_{j\to\infty} Q(f_j) \leq Q(f) = Q(f/\|f\|_p)\|f\|_p^2 \leq Q(f/\|f\|_p) \leq m.
   $$
   In particular $f\neq 0$ is a maximizer of $Q$ with $\|f\|_p=1$.
   
   \medskip
   
   The claim (i) is clear. As to (ii), note that in the case $M=1, \mathcal K=|\mathcal K|^*$ there is even a
   Schwarz-symmetric maximizer. In fact, passing from $f$ to its Schwarz-symmetric spherical rearrangement $|f|^*$ we find
   $Q(f)\leq Q(|f|^*)$. This follows from $\|f\|_p=\||f|^*\|_p$ \cite[p.81]{LiebLoss} and Riesz' rearrangement inequality
 \cite[Theorem~3.7]{LiebLoss}
  \begin{align*}
     Q(f)  
     \leq \int_{\R^N}\int_{\R^N} |\mathcal K(x-y)| |f(x)||f(y)|\,dx\,dy 
     \leq  \int_{\R^N}\int_{\R^N} |\mathcal K|^*(x-y) |f|^*(x)|f|^*(y)\,dx\,dy 
     =  Q(|f|^*). 
   \end{align*} 
   If additionally $\mathcal K=|\mathcal K|^*$ is decreasing (i.e., strictly monotone) in the radial
   direction, then each minimizer of $Q$ is Schwarz-symmetric up to translations.
  This follows from the sharp version of Riesz' rearrangement inequality \cite[Theorem~3.9]{LiebLoss} and
  (iii) is proved as well.
 \end{proof}
  %The proof of Lemma~\ref{lem:MinimizerQ} generalizes to all space dimensions $N\in\N$.
   
  \begin{rem} \label{rem:Lemma} ~
    \begin{itemize}
      \item[(a)] The condition $\sup_{f\in S} Q(f) > 0$ is typically easy to check. For instance,
   it holds for kernel functions $\mathcal K$ that are uniformly positive definite close to the origin because
   we may choose $f$ with small support. In the case $M=1$ it is sometimes easier to check the equivalent condition
   $\sup_{\R^3} \widehat{\mathcal  K}>0$. The condition $\lim_{|h|\to 0} \|\mathcal K(\cdot+h)-\mathcal K\|_{p'/2,\infty} =  0$ holds provided that 
   $\mathcal K \in L^{p'/2,\infty}(\R^N;\R^{M\times M})$  is almost everywhere continuous. 
   In fact, for $\mathcal K_R$ as in the proof of Proposition~\ref{prop:nonvanishing}, we have
   $$
     \|\mathcal K(\cdot+h)-\mathcal K\|_{\frac{p'}{2},\infty} 
     \leq 2 \|\mathcal K-\mathcal K_R\|_{\frac{p'}{2},\infty}
     + \|\mathcal K_R(\cdot+h)-\mathcal K_R\|_{\frac{p'}{2},\infty}
   $$
   and the latter term  tends to 0 by the Dominated Convergence Theorem.
      \item[(b)] Lemma~\ref{lem:MinimizerQ} may as well be used to give an alternative existence proof for nontrivial
   solutions of autonomous (possibly nonlocal) elliptic PDEs as in \cite[Theorem~1]{BugLenzSchiSok}. In fact,
   the nonlocal PDE $P(D)u = |u|^{r-2} u$ in $\R^N$ with $P(i\xi)=m(\xi)$, $m$ real-valued and positive, is
   equivalent to $|v|^{r'-2}v = \mathcal K\ast v$ where $v:=|u|^{r-2}u$ and $\widehat{\mathcal K}(\xi)=
   m(\xi)^{-1}$.  Lemma~\ref{lem:MinimizerQ} provides a nontrivial solution to
   this problem under similar assumptions as in \cite[Theorem~1]{BugLenzSchiSok}, which gives
   a dual ground state of $P(D)u = |u|^{r-2} u$. In the Appendix we will provide the details and prove that
   any dual ground state is a ground state and vice versa, so we recover \cite[Theorem~1]{BugLenzSchiSok} under slightly weaker
   assumptions.  
%    On the other hand, a ground state is obtained by maximizing
%    the functional $\|\mathfrak K\ast w\|_r$ over the unit sphere in $L^2(\R^N)$ where $\widehat{\mathfrak
%    K}(\xi)= m(\xi)^{-1/2}$ and $w$ takes the role of $\hat w(\xi):=\sqrt{q(\xi)}\hat
%    u(\xi)$.
%    This is not covered by  Lemma~\ref{lem:MinimizerQ}, but we expect the analysis of the latter problem to be
%    similar, also for more general nonlinearities and systems of equations.
    \end{itemize}
  \end{rem}

%    Note that the proofs do not
%    work out for $p\in\{1,2\}$. In the case $p=2$ (studied in Theorem~\ref{thm:ENonlocal}) we have $\sup_S Q =
%    \sup_{\R^3} \hat Q$ and it depends on the set $\{x\in\R^3: \hat Q = \sup_{\R^3}\hat Q\}$ whether a    
%    how this is attained: Formally: nonvanishing estimate does not work, dichotomy
%    estimate does not work\\
%    $p=1$: $\mathcal K\in L^\infty$ is not (yet) sufficient for nonvanishing, but vanishing for $\mathcal K$
%    this is fine; one might use Arzelà-Ascoli for precompactness in $L^\infty$. Problem: Bounded sequences in
%    $L^1$ do not always have weakly convergent subsequences}

 \section{Proof of Theorem~\ref{thm:ENonlocalq}}
 
 For $1<q<2$ we consider $I_q(E)= \frac{1}{2} I_L(E) - \frac{1}{2q} I_{NL}(E)$ where
 $$
   I_L(E) 
   := \int_{\R^3} |\nabla\times E|^2  + |E|^2\,dx, \qquad
   I_{NL} (E) 
    := \int_{\R^3} (K\ast |E|^q)|E|^q \,dx.
 $$ 
 Then Young's convolution inequality shows that $I_q$  is continuously differentiable provided that $K\in
 L^{\frac{1}{2-q}}(\R^3)$. In this case, the Fr\'{e}chet derivative is given by
 \begin{align*}
   I_q'(E)[\tilde E] 
   &= \int_{\R^3}  (\nabla\times E) \cdot (\nabla\times \tilde E) +    E \cdot
   \tilde E \,dx - \int_{\R^3}\int_{\R^3}K(x-y)   E(x)\cdot \tilde E(x) |E(x)|^{q-2} |E(y)|^q
   \,dx\,dy
 \end{align*}
 for $E,\tilde E\in H$. We consider minimization over the Nehari manifold $\mathcal N_q = \{ E\in H :
 I_q'(E)[E]=0, E\neq 0\}$ and obtain as in the Proof of Theorem~\ref{thm:ENonlocal}    
 \begin{align*}% \label{eq:def_cq}
   \inf_{\mathcal N_q} I_q 
    = \inf_{E\in H\sm\{0\}} \sup_{t>0} I_q(tE)
    =   \inf_{E\in H\sm\{0\}} \frac{q-1}{2q} \left(\frac{I_L(E)^q}{I_{NL}(E)}\right)^{\frac{1}{q-1}}.    
 \end{align*}
 Note that the assumptions of Theorem~\ref{thm:ENonlocalq} imply $K\geq 0,K\not\equiv 0$, so $I_{NL}(E)>0$
 holds for all $E\in H\sm\{0\}$. Hence we obtain
 \begin{align*}
   \inf_{\mathcal N_q} I_q 
   &\geq  \inf_{E\in H\sm\{0\}}  \frac{q-1}{2q} \left( \frac{(\int_{\R^3}
   |E(x)|^2\,dx)^q}{\int_{\R^3}\int_{\R^3} K(x-y)|E(x)|^q|E(y)|^q\,dx\,dy}\right)^{\frac{1}{q-1}} \\
   &=   \inf_{E\in H\sm\{0\}} \frac{q-1}{2q}  \left(\frac{\|
   |E|^q\|_{2/q}^2}{Q(|E|^q)}\right)^{\frac{1}{q-1}}   \\
   &=   \inf_{E\in H\sm\{0\}} \frac{q-1}{2q}  Q\big(|E|^q/\| |E|^q \|_{2/q}^{-1}\big)^{-\frac{1}{q-1}}   \\
   &\geq  \frac{q-1}{2q} (\max_S Q)^{-\frac{1}{q-1}}  
 \end{align*}
 where  $S:=\{f\in L^{2/q}(\R^3):\|f\|_{2/q}=1\}$ denotes the unit sphere in $L^{2/q}(\R^3)$ and  
 $Q(f):= \int_{\R^3} (\mathcal K\ast f) f \,dx$. 
 Equality holds if and only if $\nabla\times E=0$ and $|E|^q$ is a multiple of some maximizer of $Q$ over $S$.
 Lemma~\ref{lem:MinimizerQ} shows that under the assumptions of Theorem~\ref{thm:ENonlocalq} there is 
 a Schwarz-symmetric maximizer $f(x)=f_0(|x|)$. So we may define 
 $E_\star(x) :=  \frac{x}{|x|}f_0(|x|)^{1/q}$. Then $E_\star$ is irrotational because of 
 $$  
   E_\star(x) = \nabla \Phi(x)\qquad\text{where }
   \Phi(x) :=  \int_0^{|x|} f_0(s)^{1/q}\,ds.
 $$ 
 Moreover,   $|E_\star(x)|^q = f_0(|x|)=f(x)$. Choosing $t_\star>0$ as the maximizer of 
 $t\mapsto I_q(tE_\star)$ we find  $E^\star:=t_\star E_\star\in\mathcal N_q$ as well as 
 $$
   I_q(E^\star)
   = \inf_{\mathcal N_q} I_q 
   = \frac{q-1}{2q} (\max_S Q)^{-\frac{1}{q-1}},
 $$
 so $E^\star$ minimizes $I_q$ over $\mathcal N_q$. This proves that $E^\star$ is a ground state solution and the
 claim is proved.  \qed

 \begin{rem} \label{rem:Thmq} ~
 \begin{itemize}
   \item[(a)] The uniqueness of ground states up to translations is an open question. 
   \item[(b)] The above proof shows that the conclusion of Theorem~\ref{thm:ENonlocalq} is true as long as
   $Q$ has a nonnegative radially symmetric maximizer. This may be the case for more general radially
   symmetric kernel functions $K$, possibly sign-changing ones. Our 
   focus on nonnegative radially symmetric  maximizers is  motivated by the fact that any such
   maximizer can be written as $|E|^q$ for some irrotational vector field $E$. It is unclear how to link 
   the maximizers of $Q$ to the ground state solutions for Nonlinear Maxwell Equations in
   nonradial situations. Those occur if the kernel function  $K$ is nonradial or if $|E(x)|^2$ is replaced
   by $V(x)E(x)\cdot E(x)$ with some periodic tensor field $V$, which is  relevant for applications in
   photonic crystals. In both cases, the existence of  ground states and whether those are irrotational is an
   open problem.
   %\cite[Theorem~2]{Burchard}
   \item[(c)] In the case $q>2$ the approach presented above fails because the infimum over the Nehari
   manifold is zero for all relevant kernel functions. For instance, consider any kernel function $K$ that is
   positive near the origin. Then choose a sequence $(E_n)$ that is bounded in $H^1(\curl;\R^3)$, has small support and
   satisfies $\|E_n\|_q\to\infty$, say $E_n(x):=
   \frac{x}{|x|}|x|^{-\frac{3}{q}(1-\frac{1}{n})} \ind_{|x|\leq \eps}$ where $K(z)\geq
   \mu>0$ for $|z|<2\eps$. Then a straightforward computation shows that the infimum over the corresponding
   Nehari manifold is zero because $(I_L(E_n))$ is bounded whereas
   $I_{NL}(E_n) \geq \mu \|E_n\|_q^{2q} \nearrow +\infty$. In particular,
   it does not make sense to look for nontrivial solutions using this approach.  
 \end{itemize}
 \end{rem}

\section{Proof of Theorem~\ref{thm:Nonlocal}}

  We present a dual variational approach for the fully nonlocal nonlinear Maxwell
  equation~\eqref{eq:FullyNonlocal} given by 
  $$
    \nabla\times \nabla\times E +   E = K\ast (|E|^{r-2}E) \quad\text{in }\R^3 
  $$
  where $r:=2q$ and the kernel function $K\in L^{r/2,\infty}(\R^3)$ satisfies $K(\cdot+h)\to K$ in
  $L^{r/2,\infty}(\R^3)$ as $|h|\to 0$. We are interested in nontrivial solutions of this problem
  that turn out to exist for all $r\in (2,\infty)$. In particular, the most important case of a Kerr
  nonlinearity $r=4$ is covered. The above equation is not variational, so  it does not make sense to look
  for ground states. The idea is to follow a dual variational approach, i.e., to consider the equation as a
  variational problem for the new unknown $U:= |E|^{r-2}E$ that is obtained after inverting the linear
  operator $E\mapsto \nabla\times\nabla\times E + E$ in suitable Lebesgue spaces. We are looking for
  solutions $E\in L^r(\R^3;\R^3)$, so the dual variational approach dealing with $U=|E|^{r-2}E$  is set up in
  $L^{r'}(\R^3;\R^3)$. From $2<r<\infty$ we infer $1<r'<2$. As mentioned above, we first need to invert the
  linear operator. We will need the following auxiliary result about the Helmholtz Decomposition that decomposes a
  vector field as a sum of its divergence-free (solenoidal) and curl-free (irrotational) part.

  \begin{prop} \label{prop:HelmholtzDecomposition}
  Assume $1<t<\infty$. Then there is a continuous projector $\Pi:L^t(\R^3;\R^3)\to L^t(\R^3;\R^3)$  such
  that $E=E_1+E_2$ with $E_1:=\Pi E, E_2:= (\id-\Pi)E$ implies $\nabla\cdot E_1=0$ and $\nabla\times E_2=0$
  in the distributional sense. In particular,
  \begin{equation*}% \label{eq:HelmholtzDec}
    \nabla\times\nabla\times E_1 = -\Delta E_1,\qquad\quad  \nabla\times\nabla\times E_2 =0.
  \end{equation*}
  \end{prop}

  The proof is based on the vector calculus identity 
  $\nabla\times\nabla\times E = - \Delta E + \nabla (\nabla\cdot E)$ and the explicit definition 
  $$
    \widehat{\Pi E}(\xi):= (1-R(\xi))\hat E(\xi)
    \qquad\text{where } 
    R(\xi):= |\xi|^{-2} \xi\xi^T \in \R^{3\times 3}.
  $$ 
  This operator indeed defines a projector on $L^{r'}(\R^3;\R^3)$ because of $R(\xi)^2=R(\xi)$ and the
  $L^{r'}\to L^{r'}-$boundedness of Riesz transforms $f \mapsto \mathcal F^{-1}(\xi_j|\xi|^{-1}\hat f)$ for
  $j=1,2,3$, see~\cite[Corollary~5.2.8]{Graf}. We now use the Helmholtz Decomposition to derive the equivalent
  dual formulation of~\eqref{eq:FullyNonlocal}. In fact, distributional solutions $E\in L^r(\R^3;\R^3)$
  of~\eqref{eq:FullyNonlocal} solve  
  $$
    (-\Delta + 1)\Pi E = \Pi [K\ast U],\qquad
      (\id-\Pi) E = (\id-\Pi) [K\ast U]   
  $$
  where $U:=|E|^{r-2}E \in L^{r'}(\R^3;\R^3)$. In Fourier variables this may rewritten as
  $$
    (|\xi|^2+1) (1-R(\xi))\hat E(\xi) =  \hat K(\xi) (1-R(\xi)) \hat U(\xi),\qquad 
     R(\xi)\hat E(\xi) =  \hat K(\xi) R(\xi) \hat U(\xi). 
  $$
  We stress that the symbols $|\xi|^2+1, \hat K(\xi), R(\xi)$ commute because the former two are
  scalar. Hence,
  \begin{equation*}% \label{eq:def_cK}
    \hat E(\xi) = \widehat{\mathcal K}(\xi) \hat U(\xi)
    \quad\text{where } \mathcal K:=\mathcal F^{-1}\left(\hat K(\cdot) \left( 
    \frac{1-R(\cdot)}{|\cdot|^2+1}+   R(\cdot) \right)\right).  
  \end{equation*}
%   In the case $\lambda<0$ the inversion of the Fourier symbol $|\xi|^2+\lambda$ requires a more
%   sophisticated analysis. We define 
%   $$
%     \mathcal F^{-1}\left( \frac{\hat g(\cdot)}{|\cdot|^2+\lambda+i0}\right)
%     := \lim_{\eps\to 0^+} \mathcal F^{-1}\left( \frac{\hat g(\cdot)}{|\cdot|^2+\lambda+i\eps}\right)
%     = \lim_{\eps\to 0^+} (-\Delta+\lambda+i\eps)^{-1}g 
%   $$
%   and this limit can be proved to exist in $L^q(\R^3;\C^3)$ under for $g\in
%   L^p(\R^3;\R^3)$ satisfying $\frac{1}{2}\leq \frac{1}{p}-\frac{1}{q}\leq
%   \frac{2}{3}$ and $\frac{1}{p},\frac{1}{q'}>\frac{2}{3}$, see \cite[Theorem~6]{Gut}. This range for $p,q$ is
%   in fact optimal in the scale of Lebesgue spaces. So we may construct (just) one solution $E$
%   of~\eqref{eq:def_cK} by solving
%   \begin{equation} \label{eq:def_cK2}
%     \hat E(\xi) = \widehat{\mathcal K}(\xi) \hat U(\xi)
%     \quad\text{where } \mathcal K:= \Real\left(\mathcal F^{-1}\left(\hat K(\cdot)\left( 
%     \frac{1-R(\cdot)}{|\cdot|^2+\lambda+i0}+  \frac{R(\cdot)}{\lambda}\right)\right)\right) \qquad\text{if
%     }\lambda<0.
%   \end{equation} 
  Plugging in $E=|U|^{r'-2}U$ and applying the inverse Fourier transform in these equations it remains to
  find a solution $U\in L^{r'}(\R^3;\R^3)$ of the integral equation
  $$
    |U|^{r'-2}U = \mathcal K \ast U\qquad\text{in }\R^3. 
  $$
  Given that $\mathcal K$ is a real-valued and symmetric tensor field, this equation has
  a variational structure. The associated energy  functional reads
  \begin{equation} \label{eq:functionalJdual}
    J(U) := \frac{1}{r'} \int_{\R^3} |U|^{r'} \,dx 
    - \frac{1}{2} \int_{\R^3} (\mathcal K\ast U)\cdot U\,dx.
  \end{equation}
  We establish the relevant properties of $\mathcal K$ with regard to Lemma~\ref{lem:MinimizerQ}. 
   
  \begin{prop}\label{prop:K}
    Assume $2<r<\infty$ and $K\in L^{r/2,\infty}(\R^3)$ and $K(\cdot+h)\to K$ in $L^{r/2,\infty}(\R^3)$ as
    $|h|\to\infty$.
    Then $\mathcal K\in L^{r/2,\infty}(\R^3;\R^{3\times 3})$ and $\mathcal K(\cdot+h)\to\mathcal
    K$ in $L^{r/2,\infty}(\R^3;\R^{3\times 3})$ as $|h|\to\infty$. Moreover, if $\int_{\R^3} (K\ast
    f)f\,dx>0$ holds for some $f\in L^{r'}(\R^3)$, then $\int_{\R^3} \mathcal K\ast F\cdot F\,dx>0$ for some $F\in L^{r'}(\R^3;\R^3)$.
  \end{prop}
  \begin{proof}
    We use the Mikhlin-H\"ormander multiplier Theorem from \cite[Theorem 6.2.7]{Graf}, which says that
    the linear operator $f\mapsto \mathcal F^{-1}(m\hat f)$ is bounded on $L^t(\R^N;\C),1<t<\infty$ provided
    that $|\partial^\alpha m(\xi)|\leq C(\alpha,N) |\xi|^{-|\alpha|}$ for all $\xi\in\R^N$ and multi-indices
    $\alpha\in\N_0^N$. We actually need a consequence of this result that   
    $f\mapsto \mathcal F^{-1}(m\hat f)$ is bounded on the Lorentz space $L^{t,s}(\R^3;\C^3),1<t<\infty,1\leq
    s\leq \infty$ provided that the tensor field $m:\R^3\to\C^{3\times 3}$ with entries $m_{ij}$ for $i,j\in\{1,2,3\}$ satisfies 
    $|\partial^\alpha m_{ij}(\xi)|\leq  C(\alpha)  |\xi|^{-|\alpha|}$ for all $\xi\in\R^3$ and
    multi-indices $\alpha\in\N_0^3$. This follows by real interpolation from the classical $L^p$-version of
    this theorem, see Theorem~1.6 and Example~1.27 in~\cite{Lunardi}.
    Using this fact for $m(\xi)= (|\xi|^2+1)^{-1}(1-R(\xi)) + R(\xi)$ we find $\|\mathcal
    K\|_{r/2,\infty} \leq C\|K\|_{r/2,\infty} < \infty$. Similarly, $\|\mathcal K(\cdot+h)-\mathcal
    K\|_{r/2,\infty} \leq C\|K(\cdot+h)- K\|_{r/2,\infty} \to 0$ as $|h|\to 0$.
      Finally, 
    $\int_{\R^3} (K\ast f)f\,dx>0$ for some $f\in L^{r'}(\R^3)$ implies that the vector field $F\in
    L^{r'}(\R^3;\R^3)$ given by $\hat F(\xi)=\frac{\xi}{|\xi|}\hat f$ satisfies $(1-R(\xi))\hat F(\xi)=0$ and
    thus 
    $$ 
      \int_{\R^3} (\mathcal K\ast F)\cdot F\,dx 
      = \int_{\R^3} \widehat{\mathcal K} \hat F\cdot \hat F\,d\xi 
      = \int_{\R^3} \hat K |\hat F|^2\,d\xi
      = \int_{\R^3} \hat K |\hat f|^2\,d\xi
      = \int_{\R^3} (K\ast f)f \,dx
      >0.
    $$ 
  \end{proof}
  
  \textbf{Proof of Theorem~\ref{thm:Nonlocal}:}
  In view of $\mathcal K\in L^{r/2,\infty}(\R^3;\R^{3\times 3})$ and Young's convolution 
  inequality the functional $J:L^{r'}(\R^3;\R^3)\to\R$ from~\eqref{eq:functionalJdual} is continuously
  differentiable with Fr\'{e}chet derivative 
  $$
      J'(U)[\tilde U] = \int_{\R^3} |U|^{r'-2}U\cdot \tilde U\,dx 
      - \int_{\R^3} (\mathcal K \ast U)\cdot  \tilde U \,dx 
  $$
  for all $U,\tilde U\in L^{r'}(\R^3;\R^3)$. To find a ground state we minimize $J$ over the associated
  Nehari manifold $\mathcal M  = \{U\in L^{r'}(\R^3;\R^3) : J'(U)[U]=0,U\neq 0\}$.  
  We find as before 
     \begin{equation*}%\label{eq:minimax}
      \inf_{\mathcal M} J
      = \inf_{U\in L^{r'}(\R^3;\R^3)\sm \{0\}} \sup_{t>0} J(tU) 
      = \inf_{U\in L^{r'}(\R^3;\R^3)\sm \{0\}} \frac{2-r'}{2r'} \left(\frac{ ( \int_{\R^3} 
      |U|^{r'}\,dx)^{\frac{2}{r'}}}{ \int_{\R^3}  (\mathcal K\ast U)\cdot U\,dx} \right)^{\frac{r'}{2-r'}}.
    \end{equation*}
  Combining the assumptions of Theorem~\ref{thm:Nonlocal} with Proposition~\ref{prop:K} we obtain that the
  assumptions of Lemma~\ref{lem:MinimizerQ} hold. Hence, the   functional $Q(U):= \int_{\R^3} 
  (\mathcal K\ast U)\cdot U\,dx$ has a maximizer $U_\star\in S$ over the unit sphere $S:= \{U\in L^{r'}(\R^3;\R^3):
  \|U\|_{r'}=1\}$. We thus obtain
   \begin{align*}%\label{eq:minimax}
      \inf_{\mathcal M} J
      = \inf_{U\in L^{r'}(\R^3;\R^3)\sm \{0\}} \frac{2-r'}{2r'} Q(U/\|U\|_{r'})^{-\frac{r'}{2-r'}} 
      \geq    \frac{2-r'}{2r'} Q(U_\star)^{-\frac{r'}{2-r'}} 
      =  J(t_\star U_\star)
   \end{align*}
   where $t_\star>0$ maximizes $t\mapsto J(tU_\star)$ so that $t_\star U_\star\in\mathcal M$. Hence,
   $U^\star:= t_\star U_\star$ is a ground state for $J$ and $E^\star:= |U^\star|^{r'-2}U^\star$ is a dual
   ground state of~\eqref{eq:FullyNonlocal}. This finishes the proof of
   Theorem~\ref{thm:Nonlocal}.\qed

\section*{Appendix -- Ground states vs. dual ground states}

We show that in several contexts the notion of a dual ground state solution coincides with the classical
notion of a ground state provided that both notions make sense for the equation under investigation. In
particular, dual ground states may be seen as reasonable substitutes for ground states.
As a model example we consider  $P(D)u=|u|^{r-2}u$ on $\R^N$ from
Remark~\ref{rem:Lemma}.
The formally equivalent dual formulation is $|v|^{r'-2}v = P(D)^{-1}v$ for $v:=|u|^{r-2}u$.
The energy functional $I:H^s(\R^N)\to\R$ and the dual energy functional $J:L^{r'}(\R^N)\to\R$ are given by 
\begin{align*}
  I(u) &= \frac{1}{2} \int_{\R^N} m(\xi)|\hat u(\xi)|^2\,d\xi - \frac{1}{r}\int_{\R^N} |u|^r\,dx, \\ 
  J(v) &= \frac{1}{r'} \int_{\R^N} |v(x)|^{r'} \,dx - \frac{1}{2} \int_{\R^N} m(\xi)^{-1} |\hat
  v(\xi)|^2\,d\xi. 
\end{align*}
Here, $m(\xi)=P(i\xi)$ is the symbol associated with $P$ and we shall assume $|\partial^\alpha
(m(\xi)^{-1})|\leq C_\alpha |\xi|^{-|\alpha|}(1+|\xi|)^{-2s}$ for $0<s<\frac{N}{2}$ such that 
$2<r<\frac{2N}{N-2s}$ and all
multi-indices $\alpha\in\N_0^N$. This is slightly stronger than Assumption~1 in~\cite{BugLenzSchiSok}. 
We now show on an abstract level that a ground state solution for $I$ is the same as a dual ground state
solution. 
 
 \medskip
 
Let $\Omega$ be a set and assume that for all $x\in\Omega$ the functions $F(x,\cdot),G(x,\cdot)\in
C^2(\R)$ are real-valued  with $F_u(x,\cdot)^{-1}=G_v(x,\cdot)$ and $F(x,0)=G(x,0)=0$
for all $x\in\Omega$. Moreover assume that there are Banach spaces $X,Y$ consisting of real-valued functions
defined on $\Omega$ and that there are
 continuously differentiable functionals $I:X\to\R,J:Y\to\R$ given by 
$$ 
  I(u) = \frac{1}{2} Q_1(u,u) - \varphi(F(\cdot,u)),\qquad 
  J(v) = \varphi(G(\cdot,v))   -  \frac{1}{2} Q_2(v,v) 
$$ 
where $Q_1:X\times X\to\R$ and $Q_2:Y\times Y\to\R$ are continuous bilinear forms and $\varphi$ is a linear
functional acting on real-valued functions defined on $\Omega$ such as $x\mapsto F(x,u(x))$ and $x\mapsto
G(x,v(x))$ for $u\in X,v\in Y$. We assume that $I,J$ are continuously differentiable with  
$$
  I'(u)[h_1] = Q_1(u,h_1) - \varphi(F_u(\cdot,u)h_1),\quad
  J'(v)[h_2] = \varphi (G_v(\cdot,v)h_2) - Q_2(v,h_2)\qquad (h_1\in X,h_2\in Y).
$$
We moreover assume that $I$ and $J$ are dual to each other in the following sense: there are subsets
$M\subset X$ and $N\subset Y$ such that the Euler-Lagrange equations for $I|_M,J|_N$ are equivalent, namely 
\begin{equation}\label{eq:(H)}
  I'(u)=0,\;  u\in M \quad\Leftrightarrow\quad  J'(v)=0,\;  v\in N
  \qquad\text{for } v(x)=F_u(x,u(x)), u(x)=G_v(x,v(x)).
\end{equation}
Note that the above example satisfies \eqref{eq:(H)} for $M=X=H^s(\R^N)$ and $N=Y=L^{r'}(\R^N)$.
In fact, the implication from left to right follows from the fractional Sobolev
Embedding $H^s(\R^N)\hookrightarrow L^r(\R^N)$ for $2\leq  r\leq \frac{2N}{N-2s}$. The opposite implication,
which should be seen as a regularity result, follows from bootstrapping $|v|^{r'-2}v=P(D)^{-1}v$ with
the aid of $\|P(D)^{-1} f\|_q \leq C\|f\|_p$ for
$0<\frac{1}{p}-\frac{1}{q}<\frac{2s}{N}$. These estimates follows from  Bessel potential estimates and
Mikhlin's multiplier Theorem, which is applicable due to $|\partial^\alpha
(m(\xi)^{-1})|\leq C_\alpha |\xi|^{-|\alpha|}(1+|\xi|)^{-2s}$.
Here one needs the sharp inequality $r<\frac{2N}{N-2s}$ to show that $v\in L^{r'}(\R^N)$ implies $v\in
L^2(\R^N)$ (via bootstrapping) and hence $u=|v|^{r'-2}v=P(D)^{-1}v\in H^s(\R^N)$.

\medskip 

We say that a ground state for $I|_M$ is a nontrivial solution of $I'(u)=0,u\in M$ with least
energy $I$ among all nontrivial critical points of $I|_M$. A dual ground state for $I|_M$ with respect to
$J|_N$ is a nontrivial solution of $I'(u)=0,u\in M$ such that the dual function $v(x):=F_u(x,u(x))$ is a
ground state for $J|_N$. 

\begin{prop}\label{prop:GS}
  Under the assumptions from above:
  \begin{itemize}
    \item[(i)] A solution $u^*\in M$ is a ground state for $I|_M$ if and only if 
    $\tilde u\in M, I'(\tilde u)=0$ implies 
    $$
        \varphi\left( F_u(\cdot,u^*)u^* - 2F(\cdot,u^*)\right) \leq 
      \varphi\left(F_u(\cdot,\tilde u)\tilde u - 2F(\cdot,\tilde u)\right). 
    $$ 
   \item[(ii)] A solution $v^*\in N$ is a ground state for $J|_N$ if and only if  
   $\tilde v\in N, J'(\tilde v)=0$ implies 
    $$
        \varphi\left( 2G(\cdot,v^*)- G_v(\cdot,v^*)v^* \right)  \leq  
        \varphi\left( 2G(\cdot,\tilde v)- G_v(\cdot,\tilde v)\tilde v \right). 
   $$ 
  \end{itemize}
\end{prop}
\begin{proof}
  Part (i)  follows from $2I(u) = 2I(u)-I'(u)[u] =  \varphi\left( F_u(\cdot,u)u - 2F(x,u) \right)$ for all
  $u\in M$ such that $I'(u)=0$. Part  (ii) is proved analogously.
\end{proof}
 
  \begin{thm}
    Under the assumptions from above:  $u^*\in M$ is a ground state for $I|_M$ if and only if it is a dual
    ground state for $I|_M$ with respect to $J|_N$.
  \end{thm}
  \begin{proof}
    Assume that $u^*\in M$ is a ground state for $I|_M$, define $v^*:= F_u(\cdot,u^*)\in N$. We have to show
    that $v^*$ is a dual ground state. So take any $v\in N$ such that $J'(v)=0$ and define $u:=G_v(\cdot,v)$.
    By \eqref{eq:(H)} we have $u\in M$ and $I'(u)=0$. Since $u^*$ is a ground state, we know from
    Proposition~\ref{prop:GS} 
    $$ 
      \varphi \left(  F_u(\cdot,u^*)u^* - 2F(\cdot,u^*)\right) 
      \leq \varphi \left(  F_u(\cdot,u)u - 2F(\cdot,u)\right).
   $$
   Plugging in $u^*=G_v(\cdot,v^*),u=G_v(\cdot,v)$ we find
   \begin{equation}\label{eq:energy}
    \varphi \left( v^*G_v(\cdot,v^*) - 2F(\cdot,G_v(\cdot,v^*))\right) \leq  
    \varphi\left( vG_v(\cdot,v)- 2F(\cdot,G_v(\cdot,v))\right).
   \end{equation}
   Since $F_u(x,\cdot)$ and $G_v(x,\cdot)$ are inverses of each other, we have for all $x\in\Omega,z\in\R$
   $$
     F(x,G_v(x,z))
     = \int_0^{G_v(x,z)} F_u(x,s)\,ds
     %= \int_0^z F_u(x,G_v(x,t)) G_{vv}(x,t)\,dt
     = \int_0^z t G_{vv}(x,t)\,dt
     = zG_v(x,z)-G(x,z).
   $$
   Using this identity for $z=v^*(x),z=v(x)$, respectively, we obtain from~\eqref{eq:energy} 
   $$
    \varphi\left( 2G(\cdot,v^*)- v^*G_v(\cdot,v^*)\right) 
    \leq \varphi\left( 2G(\cdot,v)-  vG_v(\cdot,v)\right).  
  $$    
  Given that $v,v^*\in N$ are critical points of $J$ and $v$ was arbitrary, this means that $v^*$ is a ground
  state for $J|_N$. Hence, by definition, $u^*$ is a dual ground of $I|_M$ with respect to $J_N$. 
  In a similar way one shows that a dual ground state with respect to $J|_N$ yields a ground state solution.
\end{proof}

\section*{Acknowledgments}

Funded by the Deutsche Forschungsgemeinschaft (DFG, German Research Foundation) -- Project-ID 258734477 --
SFB~1173.

\bibliographystyle{abbrv}
\bibliography{biblio}

\end{document}